\documentclass[11pt,reqno]{amsart}
\usepackage{amsmath}
\usepackage{amssymb}
\usepackage{amsthm}
\usepackage{graphicx}
\usepackage{mathrsfs}
\graphicspath{{picture/}}
\usepackage[margin=1in]{geometry}
\usepackage[colorlinks=true,allcolors=black]{hyperref}
\usepackage{enumerate}
\usepackage{eufrak}
\usepackage{slashed}
\usepackage{tikz}
\usepackage{indentfirst}
\usepackage{mathabx}
\title[LARGE-DATA INVERSE SCATTERING]{RECOVERY OF SCHR\"ODINGER NONLINEARITY FROM THE LARGE-DATA SCATTERING BEHAVIOR}
\author{YI SUN}
\address[Y. Sun]{Academy of Mathematics and Systems Science, The Chinese Academy of Sciences, Beijing 100190, CHINA}
\email{sunyi@amss.ac.cn}
\numberwithin{equation}{section}
\newtheorem{thm}{Theorem}
\numberwithin{thm}{section}
\newtheorem{prop}{Proposition}
\numberwithin{prop}{section}

\numberwithin{cor}{section}
\newtheorem{lem}{Lemma}
\numberwithin{lem}{section}
\theoremstyle{definition}
\newtheorem{mydef}{Definition}
\numberwithin{mydef}{section}

\theoremstyle{remark}
\newtheorem{rmk}{Remark}
\numberwithin{rmk}{section}

\renewcommand{\L}{\mathrm{L}}
\renewcommand{\i}{\mathrm{i}}
\renewcommand{\H}{\mathrm{H}}
\renewcommand{\S}{\mathrm{S}}

\newcommand{\I}{\mathrm{I}}
\newcommand{\e}{\mathrm{e}}
\newcommand{\W}{\mathrm{W}}

\begin{document}
	\begin{abstract}
		In this article, we study scattering and inverse scattering problems for 3D nonlinear Schr\"odinger equations of the form,
		$$\i\partial_t u + \Delta u = a(x)|u|^2u.$$
		For $a(x)$ in a suitable class, such equations are globally well-posed and admit a large-data scattering theory. In this work, we prove that the large-data scattering behavior uniquely determines the inhomogeneity $a(x)$.
	\end{abstract}
	
	\maketitle
	
	\section{Introduction}
	\subsection{Background and basic setting}
	We consider nonlinear Schr\"odinger equations of the form
	\begin{equation}\label{cubic-nls-1}
		\left\{\begin{aligned}
			\i\partial_t u + \Delta u &= a(x)|u|^2u,\\
			u(0,x) &= u_0(x) \in \H^1_x(\mathbb{R}^3),
		\end{aligned}\right.
	\end{equation}
	where $u(t,x): \mathbb{R}\times \mathbb{R}^3 \to \mathbb{C}$ and $a(x):\mathbb{R}^3\to [0,\infty)$ is a continuous function with $a, \nabla a\in \L^{\infty}_x(\mathbb{R}^3)$, and satisfies $x\cdot \nabla a \leq 0$.
	
	The equation \eqref{cubic-nls-1} enjoys several conservation laws.
	\begin{itemize}
		\item The gauge symmetry $u(t,x) \mapsto \e^{\i \theta}u(t,x)$ for $\theta \in \mathbb{R}$ corresponds to the conservation of the \textit{mass}, defined by
		\begin{equation}\label{conservation-mass}
			\mathcal{M}[u(t)] = \int_{\mathbb{R}^3} |u(t,x)|^2\ \mathrm{d}x.
		\end{equation}
		\item The time translation symmetry $u(t,x) \mapsto u(t+t_0,x)$ for $t_0\in \mathbb{R}$ corresponds to the conservation of \textit{energy}, defined by
		\begin{equation}\label{conservation-energy}
			\mathcal{E}[u(t)] = \tfrac{1}{2}\int_{\mathbb{R}^3} |\nabla u(t,x)|^2\ \mathrm{d}x + \tfrac{1}{4}\int_{\mathbb{R}^3} a(x)|u(t,x)|^4\ \mathrm{d}x.
		\end{equation}
	\end{itemize}
	
	\eqref{cubic-nls-1} has been studied in some previous works \cite{M-2023, S-1974, M. W-2018, R. W-2001-1, R. W-2001-2}. These scattering and inverse-scattering problems have been further studied in \cite{C-M-2024-1, C-M-2024-2, C-M-2025, K-M-V-2023, K-M-V-2025}. These works handled the following problems:
	\begin{enumerate}[\upshape(1)]
		\item Global well-posedness of the scattering solutions for a suitable class of initial data;
		\item Recovery of the nonlinearity from the behavior of the scattering map.
	\end{enumerate}
	
	However, previous works mostly concentrated on the small-data scattering case. In \cite{M-2023}, Murphy established the small-data scattering in $\H^1_x$ in the inter-critical case and recovered the nonlinearity from the scattering map in the small-data setting. It is of interest to extend this result to the large-data case.
	
	In \cite{M. W-2018}, Watanabe established a large-data $\H^1_x$ scattering theory in the 3D inter-critical regime by Morawetz estimates with coefficients $a(x)$ satisfying a \textit{repulsivity} condition. In this article, we directly cite the large-data scattering result established by Watanabe. However, Watanabe's inverse scattering result requires additional regularity and spatial decay assumptions. Our result, in the cubic case, removes these additional assumptions and establishes uniqueness of the coefficient from the local behavior of the scattering map near a nonzero scattering state.
	
	In \cite{S-1974}, Strauss demonstrated that knowledge of the scattering map suffices to determine integrals of the form
	\begin{equation}\label{key_integral}
		\int_{\mathbb{R}}\langle a|\e^{\i t \Delta}\varphi|^p \e^{\i t\Delta}\varphi, \e^{\i t\Delta}\psi \rangle\ \mathrm{d}t,
	\end{equation}
	which may be used to recover the coefficient $a(x)$ pointwise. In \cite{M-2023}, Murphy reduced the inverse problem to showing that knowledge of the integrals uniquely determines the coefficient $a(x)$. Murphy used the \textit{Born approximation}, expanding $\langle S_a(\varepsilon \varphi), \psi \rangle$ and calculating the coefficient of $\varepsilon^{p+1}$ if the nonlinearity is of the form $a|u|^p u$. For the final step, determining the coefficient $a(x)$ by knowledge of \eqref{key_integral}, Murphy specialized to the case of Gaussian data, for which the free Schr\"odinger evolution can be computed explicitly. Then the final step is completed by evaluating a Gaussian integral.
	
	The main idea of this article is to use the Born approximation in the large-data setting. In this case, we need to recover the nonlinearity from knowledge of scattering behavior near a given scattering state $u_{-}\in \H^1_x$ and $u_{-} \nequiv 0$. Thus, for $v_{-} = u_{-} + \varepsilon \varphi$, we need to expand
	$$\langle S_{a}(v_{-}) - S_{a}(u_{-}), \psi \rangle$$
	instead and extract the coefficient of $\varepsilon^3$. In this way, we find that $a(x)$ still can be uniquely determined from knowledge of integrals of the form \eqref{key_integral}. Another contribution of this article is in the final step. In this large-data case, to determine the inhomogeneity $a(x)$, we need a family of Gaussian data, translated by the free Schr\"odinger group, to test \eqref{key_integral}. Then by comparing the absolute value of \eqref{key_integral} with the absolute values of other integrals, the final step is reduced to evaluating a Gaussian integral, as in the final step of \cite{M-2023}.
	
	\subsection{Organization of the Paper}
	\begin{itemize}
		\item In section 2, we establish the global well-posedness in Theorem \ref{GWP}.
		\item In section 3, we clarify the definition of the scattering map and expand perturbation terms in preparation for proving the main theorem.
		\item In section 4, we prove our main result Theorem \ref{S-map-determine-nonlinearity}.
	\end{itemize}
	The main contribution of the paper is in section 4. The results in section 2 and section 3 can be proved based on classical theory in dispersive equations with no difficulty, and we believe these results for \eqref{cubic-nls-1} are well known. We give proofs here only to make the paper more self-contained.
	
	\subsection{Notations}
	We use $C$ to denote the universal constant and $C$ may change line by line. We write $A\lesssim B$ if $A\leq CB$ for some universal constant $C>0$. We write $A\sim B$, if both $A\lesssim B$ and $B\lesssim A$ hold.
	
	Our definition for the Fourier transform is
	$$\hat{f}(\xi) = (2\pi)^{-\frac{3}{2}}\int_{\mathbb{R}^3} f(x) \e^{-\i\xi\cdot x}\ \mathrm{d}x,$$
	so that
	$$f(x) = (2\pi)^{-\frac{3}{2}}\int_{\mathbb{R}^3} \hat{f}(\xi)\e^{\i\xi\cdot x}\ \mathrm{d}\xi\quad \text{and}\quad \int_{\mathbb{R}^3} |\hat{f}(\xi)|^2\ \mathrm{d}\xi = \int_{\mathbb{R}^3} |f(x)|^2\ \mathrm{d}x.$$
	
	The free Schr\"odinger group is defined by $\e^{\i t\Delta} := \mathcal{F}^{-1}\e^{-\i t |\xi|^2}\mathcal{F}$, where $\mathcal{F}$ denotes the Fourier transform.
	
	We use $\L^p_t\L^q_x(T\times X)$ to denote the mixed Lebesgue spacetime norm
	$$\|f\|_{\L^p_t\L^q_x(T\times X)} = \left\|\|f(t,x)\|_{\L^q_x(X)}\right\|_{\L^p_t(T)} = \left[\int_T\left(\int_X |f(t,x)|^q\ \mathrm{d}x\right)^{p/q}\ \mathrm{d}t\right]^{1/p}.$$
	When $p=q$, we abbreviate $\L^p_{t,x} = \L^p_t\L^q_x$. When $X=\mathbb{R}^3$, we abbreviate $\L^p_T\L^q_x = \L^p_t\L^q_x(T) = \L^p_t\L^q_x(T\times \mathbb{R}^3)$. Throughout this article, unless explicitly stated, $\L^p_t\L^q_x$ stands for $\L^p_t\L^q_x(\mathbb{R}\times \mathbb{R}^3)$.
	
	Fix $s\geq 0$. The inhomogeneous Sobolev space $\H^s_x(\mathbb{R}^3)$ is the completion of the Schwartz space $\mathcal{S}(\mathbb{R}^3)$ with respect to the norm
	$$\|f\|^2_{\H^s_x(\mathbb{R}^3)} = \|\langle \nabla \rangle^s f\|^2_{\L^2_x(\mathbb{R}^3)} = \int_{\mathbb{R}^3} (1 + |\xi|^2)^s|\hat{f}(\xi)|^2\ \mathrm{d}\xi < \infty.$$
	Sobolev space $\W^{1,\infty}_x(\mathbb{R}^3)$ is defined by
	$$\W^{1,\infty}_x(\mathbb{R}^3) := \left\{f\in \L^{\infty}_x(\mathbb{R}^3):\ \partial_{x_i} f \in \L^{\infty}_x(\mathbb{R}^3),\ i = 1,2,3\right\}$$
	with norm
	$$\|f\|_{\W^{1,\infty}_x(\mathbb{R}^3)} := \|f\|_{\L^{\infty}_x(\mathbb{R}^3)} + \|\nabla f\|_{\L^{\infty}_x(\mathbb{R}^3)}.$$
	\subsection{Acknowledgment}
	This research was partially supported by NSFC, Grant No. 12471232, 12288201, CAS Project for Young Scientists in Basic Research, Grant No. YSBR-031. The author is very grateful to his advisor Professor Chenjie Fan for patient and insightful guidance and helpful comments. This work has greatly benefited from his valuable suggestions and support.
	
	\section{Global well-posedness}
	In this section, we establish the global well-posedness result for \eqref{cubic-nls-1}. We use standard contraction mapping arguments based on Strichartz estimates and the energy method (see e.g. \cite{T-2006}). The only modification is to take into account the inhomogeneity.
	\begin{prop}[Strichartz estimates for Schr\"odinger, $d=3$, c.f. \cite{G-V-1989,K-T-1998,Y-1987}]\label{strichartz-est}
		Call a pair $(p,q)$ Schr\"odinger admissible for $d=3$ if $2\leq p,q\leq \infty$ and $\frac{2}{p}+\frac{3}{q} = \frac{3}{2}$. Then for any admissible pair $(p,q)$ and $(\tilde{p},\tilde{q})$, we have homogeneous Strichartz estimate
		\begin{equation}\label{homogeneous-Strichartz-estimate}
			\|\e^{\i t\Delta}u_0\|_{\L^p_t\L^q_x(\mathbb{R}\times \mathbb{R}^3)}\lesssim_{p,q}\|u_0\|_{\L^2_x(\mathbb{R}^3)},
		\end{equation}
		the dual homogeneous Strichartz estimate
		\begin{equation}\label{dual-homogeneous-Strichartz-estimate}
			\left\|\int_{\mathbb{R}}\e^{-\i s\Delta}F(s)\ \mathrm{d}s\right\|_{\L^2_x(\mathbb{R}^3)}\lesssim_{\tilde{p},\tilde{q}}\|F\|_{\L^{\tilde{p}'}_t\L^{\tilde{q}'}_x(\mathbb{R}\times\mathbb{R}^3)},
		\end{equation}
		and the inhomogeneous Strichartz estimate
		\begin{equation}\label{inhomogeneous-Strichartz-estimate}
			\left\|\int_{s<t}\e^{\i(t-s)\Delta}F(s)\ \mathrm{d}s\right\|_{\L^p_t\L^q_x(\mathbb{R}\times \mathbb{R}^3)}\lesssim_{p,q,\tilde{p},\tilde{q}}\|F\|_{\L^{\tilde{p}'}_t\L^{\tilde{q}'}_x(\mathbb{R}\times\mathbb{R}^3)}.
		\end{equation}
	\end{prop}
	\subsection{Local well-posedness}
	We construct the local solution by Duhamel formula. Let $u_0\in \H^1(\mathbb{R}^3)$. Fix $T>0$ which will be chosen later. Set
	\begin{equation}\label{Duhamel}
		u(t) = \Phi u(t) := \e^{\i t\Delta}u_0 - \i \int_{0}^{t} \e^{\i (t-s)\Delta}[a(x)(|u|^2u)(s,x)]\ \mathrm{d}s
	\end{equation}
	where $0\leq t \leq T$.
	
	Define the Strichartz norm by
	\begin{equation}\label{Strichartz-norm}
		\|u\|_{\S^1} = \sup_{\text{$(q,r)$ admissible}}\|\langle \nabla \rangle u\|_{\L^q_t\L^r_x}.
	\end{equation}
	
	Define $\mathcal{B}$ to be the set of functions $u: [0,T]\times \mathbb{R}^3 \to \mathbb{C}$ satisfying the bound
	\begin{equation}\label{B-space-bound}
		\|u\|_{\S^1}\leq 4C\|u_0\|_{\H^1_x}
	\end{equation}
	where $C$ is an absolute constant. We equip $\mathcal{B}$ with the metric
	$$\mathrm{d}(u,v) = \sup_{\text{$(q,r)$ admissible}}\|u-v\|_{\L^{q}_t\L^{r}_x}.$$
	By Strichartz estimates, H\"older's inequality and fractional Leibniz's rules,
	\begin{align}\label{B-to-B}
		&\|\Phi u\|_{\S^1} \lesssim \|u_0\|_{\H^1_x} + \|\langle \nabla \rangle a|u|^2u\|_{\L^2_t\L^{6/5}_x}\notag \\
		&\lesssim \|u_0\|_{\H^1_x} + T^{1/5}\left(\|a\|_{\W^{1,\infty}_x}\|u\|^2_{\L^{10}_t\L^{5}_x}\|u\|_{\L^{10}_t\L^{30/13}_x} + \|a\|_{\L_x^{\infty}}\|u\|^2_{\L^{10}_t\L^{5}_x}\|\langle \nabla \rangle u\|_{\L^{10}_t\L^{30/13}_x}\right)
	\end{align}
	By Sobolev embedding,
	\begin{equation}\label{Sobolev-embedding-1}
		\|u\|_{\L^{10}_t\L^{5}_x}\lesssim \|\langle \nabla \rangle u\|_{\L^{10}_t\L^{30/13}_x}.
	\end{equation}
	Thus by \eqref{B-to-B}, choosing $T$ small enough depending on $\smash{\|a\|_{\W^{1,\infty}_x}}$, $\smash{\|u_0\|_{\H^1_x}}$ and some absolute constants, we can ensure
	\begin{equation}\label{B-to-B-2}
		\|\Phi u\|_{\S^1} \leq 4C\|u_0\|_{\H^1_x}.
	\end{equation}
	For $u,v\in \mathcal{B}$, using elementary estimate
	$$\left||u|^2u - |v|^2v\right| \lesssim \left(|u|^2 + |v|^2\right)|u-v|,$$
	we obtain
	\begin{align}\label{contraction-estimate}
		\mathrm{d}\left(\Phi u,\Phi v\right) &\lesssim \|a|u|^2u - a|v|^2v\|_{\L^2_t\L^{6/5}_x}\notag \\
		&\lesssim T^{1/5}\|a\|_{\L^{\infty}_x}\left(\|u\|^2_{\L^{10}_t\L^{5}_x} + \|v\|^2_{\L^{10}_t\L^{5}_x}\right)\|u-v\|_{\L^{10}_t\L^{30/13}_x}.
	\end{align}
	By \eqref{Sobolev-embedding-1}, we can choose $T$ small enough depending on $\smash{\|a\|_{\W^{1,\infty}_x}}$, $\smash{\|u_0\|_{\H^1_x}}$ and some absolute constants so that
	\begin{equation}\label{contraction-estimate-2}
		\mathrm{d}\left(\Phi u,\Phi v\right) \leq \tfrac{1}{2}\mathrm{d}(u,v).
	\end{equation}
	Then contraction mapping principle ensures the existence and uniqueness of local solution to \eqref{cubic-nls-1}. 
	\subsection{Global well-posedness}
	Now we are ready to establish global well-posedness by the energy method. We use the Gagliardo-Nirenberg inequality:
	\begin{lem}[Gagliardo-Nirenberg inequality for $d=3$]\label{G-N-inequality-d=3-lem}
		For $u\in \H^1_x(\mathbb{R}^3)$,
		\begin{equation}\label{G-N-inequality-d=3-eq}
			\|u\|_{\L^4_x(\mathbb{R}^3)}\lesssim \|u\|^{1/4}_{\L^2_x(\mathbb{R}^3)}\|u\|^{3/4}_{\dot{\H}^1_x(\mathbb{R}^3)}.
		\end{equation}
	\end{lem}
	
	\begin{thm}\label{GWP}
		The Cauchy problem \eqref{cubic-nls-1} is globally well-posed.
	\end{thm}
	\begin{proof}
		By \eqref{conservation-mass} and \eqref{conservation-energy}, we can conclude that
		\begin{align}\label{u-H^1-est}
			\|u\|^2_{\H^1_x} &\lesssim \mathcal{M}[u(t)] + \mathcal{E}[u(t)] \sim \mathcal{M}[u_0] + \mathcal{E}[u_0]\notag \\
			&\lesssim \|u_0\|^2_{\L^2_x} + \|a\|_{\L^{\infty}_x}\|u_0\|_{\L^2_x}\|u_0\|^3_{\dot{\H}^1_x}\lesssim \|u_0\|^2_{\H^1_x}\left(1+\|a\|_{\L^{\infty}_x}\|u_0\|^2_{\H^1_x}\right).
		\end{align}
		The local result ensures that the solution can be extended globally in steps of fixed  time length $T$ and the proof of Theorem \ref{GWP} is complete.
	\end{proof}
	
	\section{The direct scattering problem}
	\subsection{Scattering of the global solutions}
	The following scattering result is established by Watanabe in \cite{M. W-2018}:
	\begin{thm}\label{large_data_sacttering}
		The global solutions to \eqref{cubic-nls-1} {\rm scatter}, which means for every global solution $u(t)$ to \eqref{cubic-nls-1}, there exist unique $u_{-}$ and $u_{+}$ in $\H^1_x(\mathbb{R}^3)$ such that
		\begin{equation}\label{scattering_def_eq}
			\lim_{t\to \pm \infty}\|u(t) - \e^{\i t\Delta}u_{\pm}\|_{\H^1_x} = 0.
		\end{equation}
	\end{thm} 
	
	\begin{rmk}\label{Strichartz_norm_finite}
		It is well-known that, by a standard bootstrap argument and estimate
		\begin{equation}\label{L5_est}
			\|\e^{\i t\Delta} \varphi\|_{\L^5_{t,x}} \lesssim \|\varphi\|_{\dot{\H}^{1/2}_x} \lesssim \|\varphi\|_{\H^1_x},
		\end{equation}
		one can obtain
		$$\|\e^{\i (t-T)\Delta}u(T)\|_{\L^5_{t,x}([T,\infty))} \to 0,\quad \text{as $T\to \infty$}.$$
		The analogous statement holds for the negative-time tail as $T\to -\infty$, and these two facts imply that
		$$\|u\|_{\L^5_{t,x}} + \|u\|_{\S^1} < \infty.$$
		It is also well-known that Strichartz estimates and $\L^5_{t,x}$ bound for $u(t)$ imply spacetime bounds for all Schr\"odinger admissible pairs
		\begin{equation}\label{strichartz-norm-finite}
			\|u\|_{\S^1} \leq C_{a,u} = C\left(\|a\|_{\W^{1,\infty}_x},\|u\|_{\L^5_{t,x}},\|u_0\|_{\H^1_x}\right).
		\end{equation}
		Throughout the rest of the article, we use $C_{a,u}$ to denote constants depending on $\|a\|_{\W^{1,\infty}_x}$, $\|u\|_{\L^5_{t,x}}$ and $\|u_0\|_{\H^1_x}$, which may change line by line.
	\end{rmk}
	The following theorem is adapted from Proposition 2.3 in \cite{H-R-2008}.
	\begin{thm}[Perturbation theorem]\label{perturbation-thm}
		For any $u_0\in \H^1(\mathbb{R}^3)$, let $u(t)$ be the corresponding global scattering solution to \eqref{cubic-nls-1}. If $v_0\in B_{\varepsilon}(u_0)$ and we set $w^{\varepsilon}(t) = v^{\varepsilon}(t) - u(t)$, where $v^{\varepsilon}(t)$ is the global scattering solution to \eqref{cubic-nls-1} corresponding to $v_0$, then
		\begin{equation}\label{w-5-norm}
			\|w^{\varepsilon}\|_{\L^5_{t,x}(\mathbb{R}\times \mathbb{R}^3)}\leq C_{a,u}\varepsilon
		\end{equation}
		and
		\begin{equation}\label{w-(p,q)-norm}
			\|\langle \nabla \rangle w^{\varepsilon}\|_{\L^p_t\L^q_x(\mathbb{R}\times \mathbb{R}^3)}\leq C_{a,u}\varepsilon
		\end{equation}
		for all Schr\"odinger admissible pairs $(p,q)$.
	\end{thm}
	\begin{proof}
		$w^{\varepsilon}(t)$ solves the equation
		$$\i \partial_t w^{\varepsilon} + \Delta w^{\varepsilon} = a|v^{\varepsilon}|^2v^{\varepsilon} - a|u|^2u.$$
		
		Introduce $\eta>0$. Split $[0,\infty)$ into $J = J(\smash{\|a\|_{\W^{1,\infty}_x},\|u\|_{\L^5_{t,x}},\|u_0\|_{\H^1_x}},\eta)$ many intervals such that on each interval $\I_j$, we have
		\begin{equation}\label{(5,30/11,1)-I_j}
			\|\langle \nabla \rangle u\|_{\L^5_{\I_j}\L^{30/11}_x} < \eta.
		\end{equation}
		Let $\I_j = [a_j,a_{j+1})$. On each $\I_j$, $w^{\varepsilon}(t)$ can be represented by
		\begin{equation}\label{w-Duhamel}
			w^{\varepsilon}(t) = \e^{\i(t-a_j)\Delta}w^{\varepsilon}(a_j) - \i \int_{a_j}^{t}\e^{\i(t-s)\Delta}[a|v^{\varepsilon}|^2v^{\varepsilon} - a|u|^2u]\ \mathrm{d}s.
		\end{equation}
		By fractional Leibniz rules,
		$$\begin{aligned}
			\|\langle \nabla \rangle [a|u+w^{\varepsilon}|^2&(u+w^{\varepsilon}) - a|u|^2u]\|_{\L^{5/3}_t\L^{30/23}_x}\lesssim \|a\|_{\W^{1,\infty}_x}\left[\|u\|^2_{\L^5_{t,x}}+\|w^{\varepsilon}\|^2_{\L^5_{t,x}}\right]\|\langle \nabla \rangle w^{\varepsilon}\|_{\L^5_{t}\L^{30/11}_x}\\
			&+\|a\|_{\W^{1,\infty}_x}\|w^{\varepsilon}\|_{\L^5_{t,x}}\left[\|u\|_{\L^5_{t,x}}+\|w^{\varepsilon}\|_{\L^5_{t,x}}\right]\left[\|\langle \nabla \rangle u\|_{\L^5_{t}\L^{30/11}_x} + \|\langle \nabla \rangle w^{\varepsilon}\|_{\L^5_{t}\L^{30/11}_x}\right]
		\end{aligned}$$
		Therefore
		$$\|\langle \nabla \rangle w^{\varepsilon}\|_{\L^5_{\I_j}\L^{30/11}_x}\leq \gamma_j + C\|a\|_{\W^{1,\infty}_x}\left(\eta^2 \|\langle \nabla \rangle w^{\varepsilon}\|_{\L^5_{\I_j}\L^{30/11}_x} + \eta\|\langle \nabla \rangle w^{\varepsilon}\|^2_{\L^5_{\I_j}\L^{30/11}_x} + \|\langle \nabla \rangle w^{\varepsilon}\|^3_{\L^5_{\I_j}\L^{30/11}_x}\right),$$
		where
		$$\gamma_j = \|\langle \nabla \rangle \e^{\i(t-a_j)\Delta}w^{\varepsilon}(a_j)\|_{\L^5_{t}\L^{30/11}_x}.$$
		If $C\eta^2\|a\|_{\W^{1,\infty}_x}\leq\tfrac{1}{3}$, we get
		$$\|\langle \nabla \rangle w^{\varepsilon}\|_{\L^5_{\I_j}\L^{30/11}_x}\leq \tfrac{3}{2}\gamma_j + \tilde{C}\|a\|_{\W^{1,\infty}_x}\left(\|\langle \nabla \rangle w^{\varepsilon}\|^2_{\L^5_{\I_j}\L^{30/11}_x} + \|\langle \nabla \rangle w^{\varepsilon}\|^3_{\L^5_{\I_j}\L^{30/11}_x}\right).$$
		By a bootstrap argument, there exists $C_0 = C_0(\tilde{C},\|a\|_{\W^{1,\infty}_x})$ such that, if $\gamma_j\leq C_0$, we have
		\begin{enumerate}[\upshape(1)]
			\item $\|\langle \nabla \rangle w^{\varepsilon}\|_{\L^5_{\I_j}\L^{30/11}_x}\leq 4\gamma_j$;
			\item $\tilde{C}\|a\|_{\W^{1,\infty}_x}\left(\|\langle \nabla \rangle w^{\varepsilon}\|^2_{\L^5_{\I_j}\L^{30/11}_x} + \|\langle \nabla \rangle w^{\varepsilon}\|^3_{\L^5_{\I_j}\L^{30/11}_x}\right)\leq 3\gamma_j$.
		\end{enumerate}
		Hence $\|\langle \nabla \rangle w^{\varepsilon}\|_{\L^5_{\I_j}\L^{30/11}_x}\leq 4\|\langle \nabla \rangle \e^{\i(t-a_j)\Delta}w^{\varepsilon}(a_j)\|_{\L^5_t\L^{30/11}_x}$.
		
		To continue the iteration, put $t=a_{j+1}$ in \eqref{w-Duhamel} and apply $\e^{\i(t-a_{j+1})\Delta}$ to obtain:
		$$\e^{\i(t-a_{j+1})\Delta}w^{\varepsilon}(a_{j+1}) = \e^{\i(t-a_j)\Delta}w^{\varepsilon}(a_j) - \i\int_{a_j}^{a_{j+1}}\e^{\i(t-s)\Delta}[a|u+w^{\varepsilon}|^2(u+w^{\varepsilon}) - a|u|^2u]\ \mathrm{d}s.$$
		The same argument yields
		$$\begin{aligned}
			\|\langle \nabla \rangle\e^{\i(t-a_{j+1})\Delta}w^{\varepsilon}(a_{j+1})\|_{\L^5_t\L^{30/11}_x}\leq &\|\langle \nabla \rangle \e^{\i(t-a_j)\Delta}w^{\varepsilon}(a_j)\|_{\L^5_t\L^{30/11}_x}+ C\eta^2\|a\|_{\W^{1,\infty}_x}\smash{\|\langle \nabla \rangle w^{\varepsilon}\|_{\L^5_{\I_j}\L^{30/11}_x}}\\
			&+ \tilde{C}\|a\|_{\W^{1,\infty}_x}\smash{\left(\|\langle\nabla\rangle w^{\varepsilon}\|^2_{\L^5_{\I_j}\L^{30/11}_x} + \|\langle\nabla\rangle w^{\varepsilon}\|^3_{\L^5_{\I_j}\L^{30/11}_x}\right)}.
		\end{aligned}$$
		Taking $\eta$ small enough, we see that $\gamma_{j+1}\leq 10\gamma_j$ provided $\gamma_j\leq C_0$. By assumption, we have $\gamma_0 \leq \varepsilon$. Iterating, we obtain $\gamma_j\leq 10^j \varepsilon$ if $\gamma_j\leq C_0$. If $\varepsilon = \varepsilon(\smash{\|a\|_{\W^{1,\infty}_x},\|u\|_{\L^5_{t,x}},\|u_0\|_{\H^1_x}})$ small enough such that $10^{J+1}\varepsilon\leq C_0$, then this always holds. We obtain
		\begin{equation}\label{w-5-30/11-norm}
			\|\langle \nabla \rangle w^{\varepsilon}\|_{\L^5_t\L^{30/11}_x([0,\infty))}\leq 4(J+1)10^{J+1}\varepsilon.
		\end{equation}
		Hence by Sobolev embedding, we have
		$$\|w^{\varepsilon}\|_{\L^5_{t,x}([0,\infty))}\leq C_{a,u}\varepsilon.$$
		The bounds can similarly be obtained on $(-\infty,0]$. Then by Duhamel formula and Strichartz estimates, for any Schr\"odinger pair $(p,q)$,
		$$\begin{aligned}
			\|\langle \nabla \rangle w^{\varepsilon}\|_{\L^p_t\L^q_x}\lesssim &\|v^{\varepsilon}(0) - u_0\|_{\H^1_x} + \|a\|_{\W^{1,\infty}_x}\left[\|u\|^2_{\L^5_{t,x}}+\|w^{\varepsilon}\|^2_{\L^5_{t,x}}\right]\|\langle \nabla \rangle w^{\varepsilon}\|_{\L^5_{t}\L^{30/11}_x}\\
			&+\|a\|_{\W^{1,\infty}_x}\|w^{\varepsilon}\|_{\L^5_{t,x}}\left[\|u\|_{\L^5_{t,x}}+\|w^{\varepsilon}\|_{\L^5_{t,x}}\right]\left[\|\langle \nabla \rangle u\|_{\L^5_{t}\L^{30/11}_x} + \|\langle \nabla \rangle w^{\varepsilon}\|_{\L^5_{t}\L^{30/11}_x}\right],
		\end{aligned}$$
		which yields \eqref{w-(p,q)-norm}.
	\end{proof}
	\begin{rmk}\label{rmk-perturbation-thm}
		Theorem \ref{perturbation-thm} implies that if $v_0$ is in a neighborhood of $u_0$, then $v_{-}$ is in a neighborhood of $u_{-}$. Precisely, if $v_0\in B_{\varepsilon}(u_0)$, then $v_{-}\in B_{C\varepsilon}(u_{-})$ where $C=C_{a,u}>0$ is a constant. By the triangle inequality,
		$$\|v_{-} - u_{-}\|_{\H^1_x}\leq \|\e^{-\i t\Delta}v^{\varepsilon}(t) - v_{-}\|_{\H^1_x} + \|\e^{-\i t\Delta}v^{\varepsilon}(t) - \e^{-\i t\Delta}u(t)\|_{\H^1_x} + \|\e^{-\i t\Delta}u(t) - u_{-}\|_{\H^1_x}.$$
		Taking $t$ sufficiently close to $-\infty$ and by \eqref{w-(p,q)-norm}, we have
		\begin{equation}\label{dist-v_{-}-u_{-}}
			\|v_{-} - u_{-}\|_{\H^1_x}\leq C_{a,u}\varepsilon.
		\end{equation}
		Conversely, by Strichartz estimates and \eqref{u_{-}-1}, the same argument in the proof of Theorem \ref{perturbation-thm} yields
		\begin{equation}\label{dist-v0-u0}
			\|v_0 - u_0\|_{\H^1_x} \leq C_{a,u}\|v_{-} - u_{-}\|_{\H^1_x}.
		\end{equation}
		Therefore if $v_{-}\in B_{\varepsilon}(u_{-})$, then $v_0\in B_{C\varepsilon}(u_0)$ where $C=C_{a,u}>0$.
	\end{rmk}
	
	\subsection{Asymptotic state and scattering map}
	By Duhamel formula \eqref{Duhamel}, forward asymptotic state is given by the implicit formula
	\begin{equation}\label{u_{+}}
		u_{+} = u_0 - \i\int_{0}^{\infty} \e^{-\i s\Delta}(a|u(s)|^2u(s))\ \mathrm{d}s.
	\end{equation}
	Substituting $u_0$ in \eqref{Duhamel}, we obtain
	\begin{equation}\label{Duhamel-u_{+}}
		u(t) = \e^{\i t\Delta}u_{+} + \i\int_{t}^{\infty}\e^{\i(t-s)\Delta}(a|u(s)|^2u(s))\ \mathrm{d}s.
	\end{equation}
	Similarly, backward asymptotic state can be expressed by
	\begin{equation}\label{u_{-}-1}
		u_{-} = u_0 + \i\int_{-\infty}^{0}\e^{-\i s\Delta}(a|u(s)|^2u(s))\ \mathrm{d}s
	\end{equation}
	and
	\begin{equation}\label{u_{-}-2}
		u_{-} = u_{+} + \i\int_{\mathbb{R}}\e^{-\i s\Delta}(a|u(s)|^2u(s))\ \mathrm{d}s.
	\end{equation}
	$u(t)$ can also be represented by
	\begin{equation}\label{Duhamel-u_{-}}
		u(t) = \e^{\i t\Delta}u_{-} - \i\int_{-\infty}^{t}\e^{\i(t-s)\Delta}(a|u(s)|^2u(s))\ \mathrm{d}s.
	\end{equation}
	The formulas above hold in $\H^1_x$ sense.
	\begin{rmk}\label{u_{pm}_determine-u}
		It is remarkable that $u(t)$ is uniquely determined by either $u_{-}$ or $u_{+}$ in the following sense: Assume $u_1(t)$, $u_2(t)$ are solutions to \eqref{cubic-nls-1} and
		$$\lim_{t\to -\infty}\|\e^{-\i t\Delta}u_1(t) - u_{-}\|_{\H^1_x} = \lim_{t\to-\infty}\|\e^{-\i t\Delta}u_2(t) - u_{-}\|_{\H^1_x} = 0$$
		for some $u_{-}\in \H^1_x$. By \eqref{Duhamel-u_{-}} and Strichartz estimates,
		\begin{equation}\label{u_{-}-unique-est}
			\begin{aligned}
				\|u_1(t) - u_2(t)\|_{\L^5_t\L^{30/11}_x((-\infty,T])} \leq C\|a\|_{\L^{\infty}_x}&\left[\|u_1\|^2_{\L^5_{t,x}((-\infty,T])} + \|u_2\|^2_{\L^5_{t,x}((-\infty,T])}\right]\\
				&\cdot\|u_1(t) - u_2(t)\|_{\L^5_t\L^{30/11}_x((-\infty,T])}.
			\end{aligned}
		\end{equation}
		Taking $T$ sufficiently close to $-\infty$ such that $\smash{\|u_1\|_{\L^5_{t,x}((-\infty,T])}},\smash{\|u_2\|_{\L^5_{t,x}((-\infty,T])}}\leq 1/2\sqrt{C\smash{\|a\|_{\L^{\infty}_x}}}$, a bootstrap argument yields $u_1(t) = u_2(t)$ for $t$ sufficiently small. Then $u_1(t) = u_2(t)$ for $t\in \mathbb{R}$ follows from uniqueness of the solution.
	\end{rmk}
	
	\begin{mydef}[Scattering map]\label{scattering_map}
		The \textit{scattering map} $S: u_{-}\mapsto u_{+}$ is defined by
		\begin{equation}\label{scattering_map_eq}
			Su_{-} := u_{+} = u_{-} - \i\int_{\mathbb{R}} \e^{-\i s\Delta}(a|u(s)|^2u(s))\ \mathrm{d}s.
		\end{equation}
	\end{mydef}
	
	\subsection{Expansion of perturbation terms}
	Now we set $v_{-} = u_{-} + \varepsilon\varphi\in B_{\varepsilon_0}(u_{-})$ for some fixed $\varphi$ with $\|\varphi\|_{\H^1_x}\leq 1$ and $0\leq \varepsilon \leq \varepsilon_0$ small enough. Set $w^{\varepsilon}(t) = v^{\varepsilon}(t) - u(t)$. $w^{\varepsilon}(t)$ solves
	\begin{equation}\label{w_varepsilon_solves}
		\begin{aligned}
			(\i \partial_t + \Delta)w^{\varepsilon} &= a|v^{\varepsilon}|^2v^{\varepsilon} - a|u|^2u = a|u+w^{\varepsilon}|^2(u+w^{\varepsilon}) - a|u|^2u\\
			&= 2a|u|^2w^{\varepsilon} + a u^2\overline{w^{\varepsilon}} + 2a u|w^{\varepsilon}|^2 + a\bar{u}(w^{\varepsilon})^2 + a|w^{\varepsilon}|^2w^{\varepsilon}.
		\end{aligned}
	\end{equation}
	Since
	$$\begin{aligned}
		\|w_{\varepsilon}(t) - \e^{\i t\Delta}(\varepsilon\varphi)\|_{\H^1_x} &= \|v^{\varepsilon}(t) - u(t) - \e^{\i t\Delta}[v_{-} - u_{-}]\|_{\H^1_x}\\
		&\leq \|v^{\varepsilon}(t) - \e^{\i t\Delta}v_{-}\|_{\H^1_x} + \|u(t) - \e^{\i t\Delta}u_{-}\|_{\H^1_x},
	\end{aligned}$$
	it turns out that
	\begin{equation}\label{w_varepsilon-scatter-eq}
		\lim_{t \to -\infty}\|w_{\varepsilon}(t) - \e^{\i t\Delta}(\varepsilon \varphi)\|_{\H^1_x} = 0.
	\end{equation}
	We next show that $w^{\varepsilon}(t)$ can be expanded as
	$$w^{\varepsilon}(t,x) = \varepsilon w_{1}(t,x) + \varepsilon^2 w_{2}(t,x) + \varepsilon^3 w_{3}(t,x) + o(\varepsilon^3).$$
	More precisely, we have the following propositions:
	\begin{prop}\label{w_1}
		\begin{equation}\label{def-w_1}
			w_1(t,x) := \lim_{\varepsilon\to 0}\frac{w^{\varepsilon}(t,x)}{\varepsilon}
		\end{equation}
		is a well-defined $\S^1$ function solving
		\begin{equation}\label{w_1-eq}
			(\i \partial_t + \Delta)w_{1} = 2a|u|^2w_{1} + au^2\overline{w_1}
		\end{equation}
		and
		\begin{equation}\label{w_1-S1-bound}
			\|w_1\|_{\S^1}\leq C_{a,u}.
		\end{equation}
	\end{prop}
	\begin{proof}
		Firstly, we prove $\left\{\tfrac{w^{\varepsilon}}{\varepsilon}\right\}$ is a Cauchy sequence in $\S^1$. By Duhamel formula, we have
		$$\begin{aligned}
			\tfrac{w^{\varepsilon_n}}{\varepsilon_n} - \tfrac{w^{\varepsilon_m}}{\varepsilon_m} = -\i \int_{-\infty}^{t} \e^{\i(t-s)\Delta}\Big[ &2a|u|^2\left(\tfrac{w^{\varepsilon_n}}{\varepsilon_n} - \tfrac{w^{\varepsilon_m}}{\varepsilon_m}\right) + a u^2 \left(\tfrac{\overline{w^{\varepsilon_n}}}{\varepsilon_n} - \tfrac{\overline{w^{\varepsilon_m}}}{\varepsilon_m}\right) + 2au\left(\tfrac{|w^{\varepsilon_n}|^2}{\varepsilon_n} - \tfrac{|w^{\varepsilon_m}|^2}{\varepsilon_m}\right)\\
			&+a\bar{u}\left(\tfrac{(w^{\varepsilon_n})^2}{\varepsilon_n} - \tfrac{(w^{\varepsilon_m})^2}{\varepsilon_m}\right) + a\left(\tfrac{|w^{\varepsilon_n}|^2w^{\varepsilon_n}}{\varepsilon_n} - \tfrac{|w^{\varepsilon_m}|^2w^{\varepsilon_m}}{\varepsilon_m}\right)\Big]\ \mathrm{d}s.
		\end{aligned}$$
		Similar to the bootstrap argument in the proof of Theorem \ref{perturbation-thm}, using Strichartz estimates and \eqref{w-(p,q)-norm}, we can obtain
		$$\left\|\tfrac{w^{\varepsilon_n}}{\varepsilon_n} - \tfrac{w^{\varepsilon_m}}{\varepsilon_m}\right\|_{\S^1} \leq C_{a,u}\left(\varepsilon_n + \varepsilon_m + \varepsilon^2_n + \varepsilon^2_m\right).$$
		Thus $\left\{\tfrac{w^{\varepsilon}}{\varepsilon}\right\}$ is a Cauchy sequence in $\S^1$ and \eqref{def-w_1} is well-defined in $\S^1$.
		
		Then \eqref{w_1-eq} follows from \eqref{w_varepsilon_solves} and \eqref{w_1-S1-bound} follows directly from \eqref{w-(p,q)-norm}.
	\end{proof}
	
	One should notice that $\varphi$ determines $w_1(t)$ in the following sense:
	\begin{prop}\label{w-scatter-varphi}
		We have
		\begin{equation}\label{w-scatter-varphi-eq}
			\lim_{t\to -\infty}\|w_1(t) - \e^{\i t\Delta}\varphi\|_{\H^1_x} = 0.
		\end{equation}
	\end{prop}
	\begin{proof}
		By formula \eqref{Duhamel-u_{-}}, we can write
		\begin{equation}\label{w_varepsilon-varphi}
			\begin{aligned}
				w^{\varepsilon}(t) &= \e^{\i t\Delta}(\varepsilon \varphi) - \i\int_{-\infty}^{t}\e^{\i(t-s)\Delta}\left[a|v^{\varepsilon}(s)|^2v^{\varepsilon}(s) - a|u(s)|^2u(s)\right]\ \mathrm{d}s\\
				&= \e^{\i t\Delta}(\varepsilon \varphi) - \i\int_{-\infty}^{t}\e^{\i(t-s)\Delta}\left[a|u+w^{\varepsilon}|^2(s)(u+w^{\varepsilon})(s) - a|u(s)|^2u(s)\right]\ \mathrm{d}s.
			\end{aligned}
		\end{equation}
		Thus $w^{\varepsilon}(t)$ can be written as:
		$$w^{\varepsilon} = \e^{\i t\Delta}(\varepsilon\varphi) + w^{\text{nl}}_{\varepsilon}.$$
		To prove \eqref{w-scatter-varphi-eq}, our aim is to prove
		\begin{equation}\label{w-scatter-varphi-aim}
			\lim_{t\to-\infty}\left\|\lim_{\varepsilon\to 0}\frac{w_{\varepsilon}^{\text{nl}}(t)}{\varepsilon}\right\|_{\H^1_x} = 0.
		\end{equation}
		Estimating as in Theorem \ref{perturbation-thm} and using Strichartz estimates, we obtain
		$$\begin{aligned}
			&\|w_{\varepsilon}^{\text{nl}}(t)\|_{\H^1_x}\lesssim \|a\|_{\W^{1,\infty}_x}\left[\|u\|^2_{\L^5_{(-\infty,t),x}}+\|w^{\varepsilon}\|^2_{\L^5_{(-\infty,t),x}}\right]\|\langle \nabla \rangle w^{\varepsilon}\|_{\L^5_{(-\infty,t)}\L^{30/11}_x}\\
			&+\|a\|_{\W^{1,\infty}_x}\|w^{\varepsilon}\|_{\L^5_{(-\infty,t),x}}\left[\|u\|_{\L^5_{(-\infty,t),x}}+\|w^{\varepsilon}\|_{\L^5_{(-\infty,t),x}}\right]\left[\|\langle \nabla \rangle u\|_{\L^5_{(-\infty,t)}\L^{30/11}_x} + \|\langle \nabla \rangle w^{\varepsilon}\|_{\L^5_{(-\infty,t)}\L^{30/11}_x}\right].
		\end{aligned}$$
		By Fatou's lemma,
		$$\lim_{t\to-\infty}\left\|\lim_{\varepsilon\to 0}\frac{w_{\varepsilon}^{\text{nl}}(t)}{\varepsilon}\right\|_{\H^1_x} \leq \lim_{t \to -\infty}\lim_{\varepsilon\to 0}\frac{\|w^{\text{nl}}_{\varepsilon}(t)\|_{\H^1_x}}{\varepsilon} = 0$$
		where the last equality follows from \eqref{strichartz-norm-finite} and \eqref{w-(p,q)-norm}.
	\end{proof}
	\begin{rmk}\label{w_1-formula}
		By Proposition \ref{w_1} and Proposition \ref{w-scatter-varphi}, we have
		\begin{equation}\label{w-formula-eq}
			w_1(t) = \e^{\i t\Delta}\varphi + w^{\text{nl}}_{1}(t)
		\end{equation}
		where
		\begin{equation}\label{w^{nl}_1}
			w^{\text{nl}}_{1}(t) = -\i \int_{-\infty}^{t}\e^{\i(t-s)\Delta}\left[2a|u|^2(s)w_{1}(s) + au^2(s)\overline{w_1}(s)\right]\ \mathrm{d}s.
		\end{equation}
		By a similar estimate in the proof of Proposition \ref{w-scatter-varphi}, we can obtain
		\begin{equation}\label{w^{nl}-S-norm-decay}
			\lim_{T \to -\infty} \|\langle \nabla \rangle w^{\text{nl}}_1(t)\|_{\L^p_t\L^q_x((-\infty,T))} = 0
		\end{equation}
		for all Schr\"odinger admissible pairs $(p,q)$.
	\end{rmk}
	The following uniform tail decay is critical to the proof in section 4.
	\begin{lem}\label{w^{nl}_uniform_decay}
		Let $T\in \mathbb{R}$, and
		\begin{equation}\label{w^{T}_1_formula}
			w^{T}_1(t) = \e^{\i (t+T)\Delta}\varphi + w^{{\rm nl},T}_1(t).
		\end{equation}
		Denote $\I_T = (-\infty,T)$. It holds
		\begin{equation}\label{w^{nl}_1_uniform_tail_bound}
			\lim_{T\to -\infty}\|w^{{\rm nl},T}_1(t)\|_{\S^1(\I_T)} = 0.
		\end{equation}
	\end{lem}
	\begin{proof}
		By Strichartz estimates and \eqref{w^{nl}_1}, we obtain
		\begin{equation}\label{w^{T}_1_est}
			\|w^{{\rm nl},T}_1(t)\|_{\S^1(\I_T)} \lesssim \|a\|_{\W^{1,\infty}_x}\left(\|w^{T}_1\|_{\L^5_{\I_T,x}} + \|\langle \nabla \rangle w^{T}_1\|_{\L^5_{\I_T}\L^{30/11}_x}\right)\left(\|u\|_{\L^5_{\I_T,x}} + \|\langle \nabla \rangle u\|_{\L^5_{\I_T}\L^{30/11}_x}\right)^2
		\end{equation}
		By well-posedness theory, we know
		$$\|w^{T}_1\|_{\S^1} \leq C_{a,u}\|\e^{\i T\Delta}\varphi\|_{\H^1_x} = C_{a,u}\|\varphi\|_{\H^1_x}.$$
		Letting $T\to -\infty$, \eqref{w^{T}_1_est} yields \eqref{w^{nl}_1_uniform_tail_bound} by the integrability of $u(t)$ in time.
	\end{proof}
	
	\begin{prop}\label{w_2}
		\begin{equation}\label{def-w_2}
			w_2(t,x) := \lim_{\varepsilon\to 0}\frac{w^{\varepsilon} - \varepsilon w_1}{\varepsilon^2}
		\end{equation}
		is a well-defined $\S^1$ function solving
		\begin{equation}\label{w_2-eq}
			(\i \partial_t + \Delta)w_{2} = 2a|u|^2w_{2} + au^2\overline{w_2} + 2au|w_{1}|^2 + a\bar{u}(w_1)^2
		\end{equation}
		with
		\begin{equation}\label{w_2-S1-bound}
			\|w_2\|_{\S^1}\leq C_{a,u},
		\end{equation}
		and
		\begin{equation}\label{w_2_scatter}
			\lim_{t\to -\infty}\|w_2\|_{\S^1((-\infty,t))} = 0.
		\end{equation}
	\end{prop}
	\begin{proof}
		By \eqref{w_varepsilon-varphi} and \eqref{w-formula-eq} and Duhamel formula, we have
		$$\tfrac{w^{\varepsilon} - \varepsilon w_{1}}{\varepsilon^2} = -\i \int_{-\infty}^{t} \e^{\i(t-s)\Delta}\left[ 2a|u|^2\left(\tfrac{w^{\varepsilon} - \varepsilon w_{1}}{\varepsilon^2}\right) + a u^2 \left(\tfrac{\overline{w^{\varepsilon}} - \varepsilon \overline{w_{1}}}{\varepsilon^2}\right) + 2au\tfrac{|w^{\varepsilon}|^2}{\varepsilon^2} + a\bar{u}\tfrac{(w^{\varepsilon})^2}{\varepsilon^2} + a\tfrac{|w^{\varepsilon}|^2w^{\varepsilon}}{\varepsilon^2}\right]\ \mathrm{d}s.$$
		By a bootstrap argument, using Strichartz estimates and \eqref{w-(p,q)-norm}, we can obtain
		$$\left\|\tfrac{w^{\varepsilon_n} - \varepsilon_n w_{1}}{\varepsilon^2_n} - \tfrac{w^{\varepsilon_m}-\varepsilon_m w_{1}}{\varepsilon^2_m}\right\|_{\S^1} \leq C_{a,u}\left(\left\|\tfrac{w^{\varepsilon_n}}{\varepsilon_n} - \tfrac{w^{\varepsilon_m}}{\varepsilon_m}\right\|_{\S^1} + \varepsilon_n + \varepsilon_m\right),$$
		which tends to $0$ as $\varepsilon_n,\varepsilon_m \to 0$. Thus $\left\{\tfrac{w^{\varepsilon} - \varepsilon w_1}{\varepsilon^2}\right\}$ is a Cauchy sequence in $\S^1$ and \eqref{def-w_2} is well-defined in $\S^1$. \eqref{w_2-eq} follows from \eqref{w_varepsilon_solves} and \eqref{w_1-eq}. We can also use Duhamel formula, bootstrap argument and Strichartz estimates to obtain
		$$\|w_{2}\|_{\S^1(\I_T)} \leq C_{a,u} \|\langle \nabla \rangle w_{1}\|^2_{\L^5_{\I_T}\L^{30/11}_x}$$
		for any time interval $\I_T = (-\infty,T)$, $T\in \mathbb{R}$, which yields \eqref{w_2-S1-bound} and \eqref{w_2_scatter}.
	\end{proof}
	\begin{rmk}\label{w_2 -formula}
		$w_2(t)$ can be represented by the formula
		\begin{equation}\label{w_2-formula-eq}
			w_2(t) = -\i\int_{-\infty}^{t} \e^{\i(t-s)\Delta}\left[2a|u|^2w_2 + au^2\overline{w_2} + 2au|w_1|^2 + a\bar{u}(w_1)^2\right]\ \mathrm{d}s.
		\end{equation}
	\end{rmk}
	\begin{prop}\label{w_3}
		\begin{equation}\label{def-w_3}
			w_3(t,x) := \lim_{\varepsilon\to 0}\frac{w^{\varepsilon} - \varepsilon w_1 - \varepsilon^2 w_2}{\varepsilon^3}
		\end{equation}
		is a well-defined $\S^1$ function solving
		\begin{equation}\label{w_3-eq}
			(\i\partial_t + \Delta)w_{3} = 2a|u|^2 w_{3} + au^2\overline{w_{3}} +2au(\overline{w_{1}}w_{2} + w_{1}\overline{w_{2}}) + 2a\overline{u}w_{1}w_{2} + a|w_{1}|^2w_{1}
		\end{equation}
		with
		\begin{equation}\label{w_3-S1-bound}
			\|w_3\|_{\S^1}\leq C_{a,u},
		\end{equation}
		and
		\begin{equation}\label{w_3_scatter}
			\lim_{t\to -\infty}\|w_3\|_{\S^1((-\infty,t))} = 0.
		\end{equation}
	\end{prop}
	\begin{proof}
		By \eqref{w_varepsilon-varphi}, \eqref{w-formula-eq} and \eqref{w_2-formula-eq}, we have
		$$\begin{aligned}
			\tfrac{w^{\varepsilon} - \varepsilon w_{1} - \varepsilon^2 w_2}{\varepsilon^3} = -\i \int_{-\infty}^{t} \e^{\i(t-s)\Delta}\Big[ &2a|u|^2\left(\tfrac{w^{\varepsilon} - \varepsilon w_{1} - \varepsilon^2 w_2}{\varepsilon^3}\right) + a u^2 \left(\tfrac{\overline{w^{\varepsilon}} - \varepsilon \overline{w_{1}} - \varepsilon^2\overline{w_2}}{\varepsilon^3}\right) + 2au\left(\tfrac{|w^{\varepsilon}|^2 - |\varepsilon w_1|^2}{\varepsilon^3}\right)\\
			&+ a\bar{u}\left(\tfrac{(w^{\varepsilon})^2 - (\varepsilon w_1)^2}{\varepsilon^3}\right) + a\tfrac{|w^{\varepsilon}|^2w^{\varepsilon}}{\varepsilon^3}\Big]\ \mathrm{d}s.
		\end{aligned}$$
		By a bootstrap argument, using Strichartz estimates and \eqref{w-(p,q)-norm}, we can obtain
		$$\left\|\tfrac{w^{\varepsilon_n}-\varepsilon_n w_{1}-\varepsilon^2_n w_2}{\varepsilon^3_n} - \tfrac{w^{\varepsilon_m}-\varepsilon_m w_{1}-\varepsilon^2_m w_2}{\varepsilon^3_m}\right\|_{\S^1} \leq C_{a,u}\left(\left\|\tfrac{w^{\varepsilon_n} - \varepsilon_n w_{1}}{\varepsilon^2_n} - \tfrac{w^{\varepsilon_m}-\varepsilon_m w_{1}}{\varepsilon^2_m}\right\|_{\S^1} + \left\|\tfrac{w^{\varepsilon_n}}{\varepsilon_n} - \tfrac{w^{\varepsilon_m}}{\varepsilon_m}\right\|_{\S^1}\right),$$
		which tends to $0$ as $\varepsilon_n,\varepsilon_m \to 0$. Thus $\left\{\tfrac{w^{\varepsilon} - \varepsilon w_1 - \varepsilon^2 w_2}{\varepsilon^3}\right\}$ is a Cauchy sequence in $\S^1$ and \eqref{def-w_3} is well-defined in $\S^1$. \eqref{w_3-eq} follows from \eqref{w_varepsilon_solves}, \eqref{w_1-eq} and \eqref{w_2-eq}.  We can also use Duhamel formula, bootstrap argument and Strichartz estimates to obtain
		$$\|w_{3}\|_{\S^1(\I_T)} \leq C_{a,u} \left(\|\langle \nabla \rangle w_{1}\|_{\L^5_{\I_T}\L^{30/11}_x}\|\langle \nabla \rangle w_{2}\|_{\L^5_{\I_T}\L^{30/11}_x} + \|\langle \nabla \rangle w_{1}\|^3_{\L^5_{\I_T}\L^{30/11}_x}\right)$$
		for any time interval $\I_T = (-\infty,T)$, $T\in \mathbb{R}$, which yields \eqref{w_3-S1-bound} and \eqref{w_3_scatter}.
	\end{proof}
	
	\section{The inverse scattering problem}
	This section is the main contribution of this article. We prove the scattering map determines the nonlinearity in the following sense:
	\begin{thm}\label{S-map-determine-nonlinearity}
		Assume $u_a(t)$ is a global scattering solution of
		\begin{equation}\label{nls-a}
			(\i \partial_t + \Delta)u_{a} = a|u_a|^2u_a,
		\end{equation} 
		and $u_b(t)$ is a global scattering solution of
		\begin{equation}\label{nls-b}
			(\i \partial_t + \Delta)u_{b} = b|u_b|^2u_b,
		\end{equation}
		where $a,b\geq 0$ are continuous, $a,b,\nabla a,\nabla b \in \L^{\infty}_x$ and $x\cdot \nabla a, x\cdot \nabla b \leq 0$. Assume
		$$\lim_{t\to -\infty}\|u_a(t) - \e^{\i t\Delta}u_{-}\|_{\H^1_x} = \lim_{t\to -\infty}\|u_b(t) - \e^{\i t\Delta}u_{-}\|_{\H^1_x} = 0$$
		for some $u_{-}\in \H^1_x$ and $u_{-}\nequiv 0$. Let $S_a, S_b$ be the scattering map corresponding to \eqref{nls-a} and \eqref{nls-b} respectively. For any $\varepsilon > 0$, if $S_a = S_b$ on $B_{\varepsilon}(u_{-})$, then $a\equiv b$.
	\end{thm}
	
	\begin{lem}\label{func-test-scattering-map}
		Define
		\begin{equation}\label{func-test-scattering-map-eq}
			f^{\varepsilon}_{a}(\varphi,\psi) := \i\langle S_a v_{-} - S_{a} u_{-} - I(v_{-} - u_{-}),\psi \rangle.
		\end{equation}
		Then
		\begin{equation}\label{func-test-scattering-map-eq-2}
			\begin{aligned}
				f^{\varepsilon}_{a}(\varphi,\psi) = &\left[\int_{\mathbb{R}}\langle F^{\varphi}_{a}(t,x), \e^{\i t\Delta}\psi \rangle\ \mathrm{d}t \right]\varepsilon + \left[\int_{\mathbb{R}}\langle G^{\varphi}_{a}(t,x), \e^{\i t\Delta}\psi \rangle\ \mathrm{d}t \right] \varepsilon^2\\
				&+ \left[\int_{\mathbb{R}}\langle H^{\varphi}_{a}(t,x), \e^{\i t\Delta}\psi \rangle\ \mathrm{d}t \right]\varepsilon^3 + o(\varepsilon^3),
			\end{aligned}
		\end{equation}
		where $\tfrac{o(\varepsilon^3)}{\varepsilon^3}\to 0$ as $\varepsilon\to 0$ and
		$$\begin{aligned}
			F^{\varphi}_{a}(t,x) &= 2a|u_a|^2 w_{a,1} + a(u_a)^2\overline{w_{a,1}},\\
			G^{\varphi}_{a}(t,x) &= 2a|u_a|^2 w_{a,2} + a(u_{a})^2\overline{w_{a,2}} + 2au_{a}|w_{a,1}|^2 + a\overline{u_{a}}(w_{a,1})^2,\\
			H^{\varphi}_{a}(t,x) &= 2a|u_a|^2w_{a,3} + a(u_a)^2\overline{w_{a,3}} + 2au_{a}(\overline{w_{a,1}}w_{a,2} + w_{a,1}\overline{w_{a,2}}) + 2a\overline{u_a}w_{a,1}w_{a,2} + a|w_{a,1}|^2w_{a,1},
		\end{aligned}$$
		and $w_{a,1}$, $w_{a,2}$ and $w_{a,3}$ correspond to \eqref{def-w_1}, \eqref{def-w_2} and \eqref{def-w_3} respectively with respect to \eqref{nls-a}.
	\end{lem}
	\begin{proof}
		$f^{\varepsilon}_{a}(\varphi,\psi)$ can be represented by
		\begin{equation}\label{func-test-scattering-map-eq-3}
			\begin{aligned}
				f^{\varepsilon}_{a}(\varphi,\psi) = &\left[\int_{\mathbb{R}}\langle F^{\varphi}_{a}(t,x), \e^{\i t\Delta}\psi \rangle\ \mathrm{d}t \right]\varepsilon + \left[\int_{\mathbb{R}}\langle G^{\varphi}_{a}(t,x), \e^{\i t\Delta}\psi \rangle\ \mathrm{d}t \right] \varepsilon^2\\
				&+ \int_{\mathbb{R}} \langle h^{\varepsilon,\varphi}_{a,1}(t,x) + h^{\varepsilon,\varphi}_{a,2}(t,x), \e^{\i t \Delta}\psi \rangle\ \mathrm{d}t,
			\end{aligned}
		\end{equation}
		where
		$$\begin{aligned}
			h^{\varepsilon,\varphi}_{a,1}(t,x) &= 2a|u_{a}|^2(w^{\varepsilon}_{a} - \varepsilon w_{a,1} - \varepsilon^2 w_{a,2}) + a(u_a)^2(\overline{w^{\varepsilon}_{a}} - \varepsilon \overline{w_{a,1}} - \varepsilon^2 \overline{w_{a,2}})\label{func-test-scattering-map-line2}\\
			h^{\varepsilon,\varphi}_{a,2}(t,x) &= 2a u_{a} (|w^{\varepsilon}_{a}|^2 - |\varepsilon w_{a,1}|^2) + a\overline{u_{a}}[(w^{\varepsilon}_a)^2 - (\varepsilon w_{a,1})^2] + a|w^{\varepsilon}_{a}|^2w^{\varepsilon}_{a}.\label{func-test-scattering-map-line3}
		\end{aligned}$$
		Since
		$$|w^{\varepsilon}_{a}|^2 - |\varepsilon w_{a,1}|^2 = (w^{\varepsilon}_{a} - \varepsilon w_{a,1})\overline{w^{\varepsilon}_{a}} + (\overline{w^{\varepsilon}_{a}} - \varepsilon \overline{w_{a,1}})(\varepsilon w_{a,1}),$$
		and
		$$(w^{\varepsilon}_{a})^2 - (\varepsilon w_{a,1})^2 = (w^{\varepsilon}_{a} - \varepsilon w_{a,1})(w^{\varepsilon}_{a} + \varepsilon w_{a,1}),$$
		it turns out that
		$$\int_{\mathbb{R}}\langle H^{\varphi}_{a}(t,x), \e^{\i t\Delta}\psi \rangle\ \mathrm{d}t = \lim_{\varepsilon\to 0} \varepsilon^{-3} \left[\int_{\mathbb{R}} \langle h^{\varepsilon,\varphi}_{a,1}(t,x) + h^{\varepsilon,\varphi}_{a,2}(t,x), \e^{\i t \Delta}\psi \rangle\ \mathrm{d}t\right],$$
		and the lemma follows.
	\end{proof}
	
	The free Schr\"odinger evolution of Gaussian data can be formulated explicitly (c.f. \cite{V-2014}):
	\begin{prop}\label{test-function-formula}
		Let
		$$g(x) = \exp\left\{-\tfrac{|x|^2}{4}\right\}.$$
		Then
		\begin{equation}\label{test-function-formula-eq}
			\e^{\i t\Delta}g(x) = (1+ \i t)^{-3/2}\exp\left\{-\tfrac{|x|^2}{4(1+ \i t)}\right\}.
		\end{equation}
	\end{prop}
	
	\begin{proof}[Proof of Theorem {\rm \ref{S-map-determine-nonlinearity}}]
		If $S_a = S_b$ on $B_{\varepsilon_0}(u_{-})$, then
		$$\langle (S_a v_{-} - S_{a} u_{-}) - (S_{b} v_{-} - S_{b} u_{-}) ,\psi \rangle = 0$$
		for any $\varphi, \psi \in \mathcal{S}$, where $v_{-} = u_{-} + \varepsilon \varphi$.
		
		By Lemma \ref{func-test-scattering-map}, we have
		\begin{equation}\label{order3_equal}
			\int_{\mathbb{R}}\langle H^{\varphi}_{a}(t,x), \e^{\i t\Delta}\psi \rangle\ \mathrm{d}t = \int_{\mathbb{R}}\langle H^{\varphi}_{b}(t,x), \e^{\i t\Delta}\psi \rangle\ \mathrm{d}t
		\end{equation}
		for any $\varphi, \psi\in \mathcal{S}$.
		
		Let $g(x) = \e^{-\frac{|x|^2}{4}}$ and $g_{x_0}(x) = g(x-x_0)$. Take $\varphi = \psi = \e^{\i t_0 \Delta}g_{x_0}(x)$ where $x_0\in \mathbb{R}^3$ and $t_0 > 0$ will be chosen later, and write
		$$w^{t_0}_{k,1} = \e^{\i (t+t_0)\Delta}g_{x_0}(x) + w^{\text{nl},t_0}_{k,1}.$$
		To simplify the notation, we denote
		$$z = z_{t_0,x_0}(t,x) := \e^{\i (t+t_0)\Delta}g_{x_0}(x)$$
		and write
		$$w_{k,1} = z + r_k,$$
		where $r_k := w^{{\rm nl},t_0}_{k,1}$, $k = a,b$. Then
		$$H^{\varphi}_k = k|z|^2z + \mathcal{E}_k,$$
		where
		\begin{equation}\label{Ek}
			\begin{aligned}
				\mathcal{E}_k = &2k|u_k|^2w_{k,3} + ku^2_k\overline{w_{k,3}} + 2ku_k(\overline{w_{k,1}}w_{k,2} + w_{k,1}\overline{w_{k,2}})\\
				&+2k\overline{u_k}w_{k,1}w_{k,2} + k(|w_{k,1}|^2w_{k,1} - |z|^2z)
			\end{aligned}
		\end{equation}
		and
		\begin{equation}\label{diff_remainder}
			\begin{aligned}
				|w_{k,1}|^2w_{k,1} - |z|^2z &= |z+r_k|^2(z+r_k) - |z|^2z\\
				&=2|z|^2r_k + z^2\overline{r_k} + 2z|r_k|^2 + \bar{z}r^2_k + |r_k|^2r_k
			\end{aligned}
		\end{equation}
		By \eqref{order3_equal}, we have
		\begin{equation}\label{key_reduce_eq}
			\int_{\mathbb{R}}\int_{\mathbb{R}^3} (a-b)|z|^4\ \mathrm{d}x\mathrm{d}t = \int_{\mathbb{R}}\int_{\mathbb{R}^3} (\mathcal{E}_b - \mathcal{E}_a)\bar{z}\ \mathrm{d}x\mathrm{d}t.
		\end{equation}
		By Sobolev embedding,
		$$\|u\|_{\L^4_{t,x}} \lesssim \|\langle \nabla \rangle u\|_{\L^4_t\L^3_x} \lesssim \|u\|_{\S^1}.$$
		By Proposition \ref{w_1}, \ref{w_2} and \ref{w_3},
		\begin{equation}\label{L4_bound}
			\|u_k\|_{\L^4_{t,x}} + \sum_{i=1}^{3}\|w_{k,i}\|_{\L^4_{t,x}} + \|z\|_{\L^4_{t,x}} + \|r_k\|_{\L^4_{t,x}} \leq M,
		\end{equation}
		where $M$ does not depend on $t_0$ and $x_0$. For any time interval $\I$, by H\"older, \eqref{Ek} and \eqref{diff_remainder},
		\begin{equation}\label{remainder_est}
			\begin{aligned}
				\left|\int_{\I}\int_{\mathbb{R}^3} \mathcal{E}_k \bar{z}\ \mathrm{d}x\mathrm{d}t\right| \lesssim &\|k\|_{\W^{1,\infty}_x}\Big\{ \|u_k\|^2_{\L^4_{\I,x}}\|w_{k,3}\|_{\L^4_{t,x}}\|z\|_{\L^4_{\I,x}} + \|u_k\|_{\L^4_{\I,x}}\|w_{k,1}\|_{\L^4_{t,x}}\|w_{k,2}\|_{\L^4_{t,x}}\|z\|_{\L^4_{\I,x}}\\
				&+ [\|z\|^2_{\L^4_{\I,x}}\|r_k\|_{\L^4_{\I,x}} + \|z\|_{\L^4_{\I,x}}\|r_k\|^2_{\L^4_{\I,x}} + \|r_k\|^3_{\L^4_{\I,x}}]\|z\|_{\L^4_{\I,x}}\Big\}.
			\end{aligned}
		\end{equation}
		We split $\mathbb{R}$ into $\I_1 = (-\infty,-t_0/2)$ and $\I_2 = [-t_0/2,\infty)$. On $\I_1$,
		$$\left|\int_{\I_1}\int_{\mathbb{R}^3} \mathcal{E}_k \bar{z}\ \mathrm{d}x\mathrm{d}t\right| \leq C_M \left[\|u_k\|_{\L^4_{\I_1,x}} + \|r_k\|_{\L^4_{\I_1,x}}\right].$$
		By the integrability of $u_k$ in time and Lemma \ref{w^{nl}_uniform_decay},
		$$\lim_{t_0\to \infty} \|u_k\|_{\L^4_{t,x}((-\infty,-t_0/2))} = \lim_{t_0\to \infty}\|r_k\|_{\L^4_{t,x}((-\infty,-t_0/2))} = 0.$$
		On $\I_2$,
		$$\left|\int_{\I_2}\int_{\mathbb{R}^3} \mathcal{E}_k \bar{z}\ \mathrm{d}x\mathrm{d}t\right|\leq C_M \|z\|_{\L^4_{\I_2,x}},$$
		and by \eqref{test-function-formula-eq}
		$$
			\|z\|_{\L^4_{\I_2,x}} = \left(\int_{t_0/2}^{\infty}(1+t^2)^{-3}\int_{\mathbb{R}^3} \e^{-\frac{|x-x_0|^2}{1+t^2}}\ \mathrm{d}x\ \mathrm{d}t\right)^{1/4}\lesssim \left(\int_{t_0/2}^{\infty} \tfrac{1}{t^3}\ \mathrm{d}t\right)^{1/4}.$$
		Therefore $\|z\|_{\L^4_{\I_2,x}} \lesssim t^{-1/2}_0$, and we can obtain
		$$\left|\int_{\mathbb{R}}\int_{\mathbb{R}^3} \left[a(x) - b(x)\right]|\e^{\i(t+t_0)\Delta}g_{x_0}(x)|^4\ \mathrm{d}x\mathrm{d}t\right| \lesssim R(t_0)$$
		where $R(t_0)\to 0$ as $t_0\to \infty$. But by the translation invariance of the Lebesgue integral,
		$$\int_{\mathbb{R}}\int_{\mathbb{R}^3} \left[a(x) - b(x)\right]|\e^{\i(t+t_0)\Delta}g_{x_0}(x)|^4\ \mathrm{d}x\mathrm{d}t = \int_{\mathbb{R}}\int_{\mathbb{R}^3} \left[a(x) - b(x)\right]|\e^{\i t\Delta}g_{x_0}(x)|^4\ \mathrm{d}x\mathrm{d}t$$
		for any $t_0\in \mathbb{R}$. Thus by letting $t_0\to \infty$, we obtain
		\begin{equation}\label{final_step_eq}
			\int_{\mathbb{R}}\int_{\mathbb{R}^3} \left[a(x) - b(x)\right]|\e^{\i t\Delta}g_{x_0}(x)|^4\ \mathrm{d}x\mathrm{d}t = 0
		\end{equation}
		for every $x_0\in \mathbb{R}^3$. Let
		$$K_{g}(x) := \int_{\mathbb{R}} |\e^{\i t\Delta}g(x)|^4\ \mathrm{d}t.$$
		Then \eqref{final_step_eq} implies
		\begin{equation}\label{final_step_eq_2}
			\int_{\mathbb{R}^3} [a(x) - b(x)]K_g(x-x_0)\ \mathrm{d}x = 0,
		\end{equation}
		for every $x_0 \in \mathbb{R}^3$. By \eqref{test-function-formula-eq}, we have
		$$K_{g}(x) = \int_{\mathbb{R}} (1+t^2)^{-3} \exp\left(-\tfrac{|x|^2}{1+t^2}\right)\ \mathrm{d}t.$$
		Since
		$$\|K_g(x)\|_{\L^1_x} = C\int_{\mathbb{R}} (1+t^2)^{-3/2}\ \mathrm{d}t < \infty,$$
		and
		$$\widehat{K}_g(\xi) = C \int_{\mathbb{R}} (1+t^2)^{-3/2}\exp\left(-\tfrac{(1+t^2)|\xi|^2}{4}\right) > 0$$
		for every $\xi \in \mathbb{R}^3$, by Wiener's Tauberian theorem and \eqref{final_step_eq_2}, we obtain $a\equiv b$.
	\end{proof}

\end{document}